\documentclass[12pt]{amsart}
\usepackage{amssymb,amsmath, amsthm, times,amscd, color}
\usepackage[utf8]{inputenc}
\newcommand{\CC}{\mathbb{C}} 

\newcommand{\PP}{\mathbb{P}}

\newcommand{\la}{\lambda}

\newcommand{\ov}{\overline}

\newtheorem{thm}{Theorem}[section]

\newtheorem{cor}[thm]{Corollary}
\newtheorem{lem}[thm]{Lemma}
\newtheorem{prop}[thm]{Proposition}
\theoremstyle{definition}
\newtheorem{definition}[thm]{Definition}
\theoremstyle{remark}

\newtheorem{example}[thm]{Example}

\numberwithin{equation}{section}
\title[linear fractional maps and fixed point theorems]{Generalizations of linear fractional maps for classical symmetric domains and related fixed point theorems for generalized balls }
\author{Yun Gao, Sui-Chung Ng and Aeryeong Seo}

\address{School of Mathematical Sciences, Shanghai Jiao Tong University, Shanghai, People’s Republic of
China}
\email{gaoyunmath@sjtu.edu.cn}

\address{School of Mathematical Sciences, Shanghai Key Laboratory of PMMP, East China Normal University,
Shanghai, People’s Republic of China}
\email{scng@math.ecnu.edu.cn}

\address{Department of Mathematics, Kyungpook National University, Daegu 41566, Republic of Korea}
\email{aeryeong.seo@knu.ac.kr}

\thanks{}%

\date\today

\begin{document}
\maketitle

\begin{abstract}

We extended the study of the linear fractional self maps (e.g. by Cowen-MacCluer and Bisi-Bracci on the unit balls) to a much more general class of domains, called \textit{generalized type-I domains}, which includes in particular the classical bounded symmetric domains of type-I and the generalized balls. Since the linear fractional maps on the unit balls are simply the restrictions of the linear maps of the ambient projective space (in which the unit ball is embedded) on a Euclidean chart with inhomogeneous coordinates, and in this article we always worked with homogeneous coordinates, here the term \textit{linear map} was used in this more general context. After establishing the fundamental result which essentially says that almost every linear self map of a generalized type-I domain can be represented by a matrix satisfying the ``expansion property" with respect to some indefinite Hermitian form, we gave a variety of results  for the linear self maps on the generalized balls, such as the holomorphic extension across the boundary, the normal form and partial double transitivity on the boundary for automorphisms, the existence and the behavior of the fixed points, etc. Our results generalize a number of known statements for the unit balls, including, for example, a theorem of Bisi-Bracci saying that any linear fractional map of the unit ball with more than two boundary fixed points must have an interior fixed point.
\end{abstract}

\section{Introduction}
The linear fractional maps from the complex plane into itself are among the very first objects of study in one-variable complex analysis since they have many good geometric properties (e.g. mapping real lines and circles among themselves) and they are also the only biholomorphisms of the unit disk and the Riemann sphere.  It is thus very natural to look at the linear fractional maps in several complex variables and such explorations have been made by, for instance, Cowen-MacCluer~\cite{CM} and Bisi-Bracci~\cite{BB}. In their works, they have focused on the linear fractional maps that map the unit ball into itself. They obtained a variety of results for those linear fractional maps, including the classification, normal forms, and also fixed point theorems for such maps. From the point of view of function theory, every linear fractional self map of the unit ball  also give rise to a composition operator in various function spaces of the unit ball and such an operator has been studied by many people, e.g. Cowen~\cite{Co}, Bayart~\cite{Ba} and Chen-Jiang~\cite{CJ}.

In this article, we will try to extend the study of linear fractional self maps to a much more general class of domains, which in particular include the classical bounded symmetric domains of type I and also the so-called \textit{generalized balls}. These two classes both contain the usual complex unit balls as special cases.  Although here our major focus is on the generalized balls, our results should shed some light on the behavior of the linear fractional maps defined on the classical bounded symmetric domains of type-I since they are determined by the same set of matrices (as we will see in Theorem~\ref{expansion matrix}).

For any two positive integers $p$ and $q$, let $H_{p,q}$ be the standard  non-degenerate
Hermitian form of signature $(p,q)$ on $\mathbb C^{p+q}$ where $p$ eigenvalues are $1$
and $q$ eigenvalues are $-1$, represented by the matrix $\begin{pmatrix}I_p&0\\0&-I_q\end{pmatrix}$ under the standard coordinates.
For a positive integer $r< p+q$, denote by $Gr(r, \mathbb C^{p+q})$ the Grassmannian of $r$-dimensional complex linear subspaces
(or simply $r$-\textit{planes}) of $\mathbb C^{p+q}$.
When $1\leq r\leq p$, we define the domain $D^r_{p,q}$ in $Gr(r, \mathbb C^{p+q})$ 
to be the set of positive definite $r$-planes in $\mathbb C^{p+q}$ with respect to $H_{p,q}$.
We call $D^r_{p,q}$ a {\it generalized type-I domain}.
The generalized type-I domain $D_{p,q}^r$ is an $SU(p,q)$-orbit on $Gr(r, \mathbb C^{p+q})$ under the natural action induced by that of $SL(p+q;\mathbb C)$ on $Gr(r,\mathbb C^{p+q})$. It is an example of \textit{flag domains} in a general context, of which one can find a comprehensive reference \cite{FHW}. Recently $D_{p,q}^r$ have been studied by Ng~\cite{ng1, ng2} regarding the proper holomorphic mappings between them
and by Kim~\cite{Kim} regarding the CR maps on some CR manifolds in $\partial D_{p,q}^r$. We remark that the case $r=p$ corresponds to the type-I classical bounded symmetric domains~\cite{ng2}. On the other extreme, when $r=1$, which will be the case of special interest to us, corresponds to the domains $D_{p,q}:=D^1_{p,q}$, which are called the generalized balls.
It follows immediately from our definition that $D_{p,q}$ can be also defined as the following domain on $\mathbb P^{p+q-1}$:
$$
D_{p,q} = \left\{ [z_1,\dots, z_{p+q}] \in \PP^{p+q-1} : |z_1|^2 +\dots + |z_p|^2 > |z_{p+1}|^2 + \dots + |z_{p+q}|^2\right\}.
$$
When $p=1$, it is biholomorphic to the unit ball in the Euclidean space $\CC^q$.

The generalized ball is one of the simplest kinds of domains on the projective space and their boundaries are smooth Levi non-degenerate (but not pseudoconvex in general) CR manifolds, of which detailed studies have been carried out by Baouendi-Huang~\cite{BH} and Baouendi-Ebenfelt-Huang~\cite{BEH}. More recently, Ng~\cite{ng2} discovered that the proper holomorphic mapping problem for the generalized balls is deeply linked to that of the classical bounded symmetric domains of type-I.

Here we recall that a linear fractional map on $\mathbb C^n$ is in fact just a restriction of a linear map on $\mathbb P^n$ expressed in terms of the inhomogeneous coordinates of a Euclidean coordinate chart $\mathbb C^n$ in $\mathbb P^n$. Similarly, if we follow our notations, a linear fractional self map of the unit ball $D_{1,q}$ simply comes from a linear map on $\mathbb P^{q}$ that maps $D_{1,q}$ into itself. We have defined our generalized type-I domains as domains on the Grassmannians on which, just like the case of the projective space, there are homogeneous coordinates and the associated linear maps. Since we will always work with homogeneous coordinates, we will thus call a self map of a generalized type-I domain a \textit{linear self map} if it is the restriction of a linear map of the ambient Grassmannian. 
\newline

\noindent\textbf{Remark.} Hence, according to our terminology, \textit{a linear self map and a linear fractional self map are the same.} ``Fractions" appear only because inhomogeneous coordinates are used.
\newline

We call a linear self map of a generalized type-I domain $D^r_{p,q}$ \textit{non-minimal} if its range is not of minimal dimension (see Definition~\ref{minimality}). In the case of the unit balls, a non-minimal linear self map is nothing but a non-constant linear self map. Roughly speaking, the minimal linear  self maps, just like the constant maps for the unit balls, are the cases for which most statements become trivialities or non-applicable. In Section~\ref{type-I domain}, we will first establish a fundamental result for the study of linear self maps of generalized type-I domains. Namely, we will prove (Theorem~\ref{expansion matrix}) that every non-minimal linear self map of a generalized type-I domain can be represented by a matrix $M$ satisfying the inequality

\begin{equation}
M^H H_{p,q} M - H_{p,q}>0, \label{expansioneq}
\end{equation}
where $H_{p,q} = \begin{pmatrix}I_p & 0 \\0 & -I_q\end{pmatrix}$, in which $I_p$, $I_q$ are identity matrices of rank $p$ and $q$, and $\cdot^H$ denotes Hermitian transposition, and ``$\geq$" means Hermitian semi-positivity. Furthermore, we will prove (Theorem~\ref{isometrythm}) that a surjective linear self map can be represented by a matrix which makes the equality hold, i.e. by a matrix in the indefinite unitary group $U(p,q)$. For the case of the unit balls, these results have been established by Cowen-MacCluer~\cite{CM} and we will modify their proof to work for any generalized type-I domain. As simple as it may look, such a matrix inequality is extremely useful in obtaining various kinds of results for the linear self maps of generalized type-I domains. In particular, we will use it to show that any non-minimal linear self map extends to a neighborhood of the closure of the domain (Theorem~\ref{extension thm}).

In Section~\ref{automorphisms}, we will study in detail the automorphism groups of the generalized balls, including the partial double transitivity on the boundary, fixed point theorems and normal forms. Here we remark that a generalized ball cannot be realized as a bounded convex domain in the Euclidean space unless it is a usual unit ball (see e.g.~\cite{gn}) and hence one cannot apply Brouwer's fixed point theorem to get a fixed point in its closure. We will establish the existence of fixed points (in the closure) in Theorem~\ref{existence of fixed points} and give a number of results regarding the behavior of the fixed points (Theorem~\ref{fixed point number thm} and~Corollaries~\ref{at most}, \ref{hs generalization}). For obtaining a normal form for automorphisms, we will show that the subgroup $U(p)\times U(q)$ and the ``non-isotropic dilations" generate the full automorphism group $Aut(D_{p,q})$ (Theorem~\ref{normal form}).

After studying the automorphisms, we will then look at arbitrary linear self maps of the generalized balls in Section~\ref{linear maps section}. We will again prove some results regarding their fixed points, including especially Theorem~\ref{bb generalization}, which is about how the number of fixed points on the boundary of a generalized ball is related to the existence of interior fixed points. This generalizes a result for the unit ball (see Bisi-Bracci~\cite{BB}) saying that any linear self map of the unit ball with more than two boundary fixed points must have an interior fixed point. We will also obtain in this section a relation between the linear self maps of the \textit{real} generalized balls of $D_{p,q}$ and those of $D_{p,q}$ (Theorem~\ref{real generalized ball}).

Finally,  in Section~\ref{examples} we will collect some illustrating or extremal examples for the results obtained in the previous sections.

\bigskip
{\bf Acknowledgements.} The first author was partially supported by the Key Program of NSFC, (No. 11531107). The second author
was partially supported by Thousand Talents Program of the Organization Department of the CPC Central Committee, and Science and Technology Commission of Shanghai Municipality (STCSM) (No. 13dz2260400). The third author was partially supported by Basic Science Research Program through the National Research 
Foundation of Korea (NRF) funded by the Ministry of Education (NRF-2019R1F1A1060175).

\section{Generalized type-I domains and their linear self maps}\label{type-I domain}

\noindent\textbf{Notations.} For what follows, for any $p,q\in\mathbb N^+$, we will equip $\mathbb C^{p+q}$ with the standard non-degenerate indefinite Hermitian form $H_{p,q}$ and denote the resulting indefinite inner product space by $\mathbb C^{p,q}$.
\newline

Denote $M_n$ the set of $n$-by-$n$ matrices with complex entries. Let $M\in M_{p+q}$ and consider the linear map, which will also be denoted by $M$, from $\mathbb C^{p,q}$ into itself, given by $z\mapsto Mz$, where $z\in\mathbb C^{p,q}$ is regarded as a column vector.
Let the null space of $M$ be $ker(M):=\{z\in\mathbb C^{p,q}:Mz=0\}$. Then the image of every positive definite $r$-plane in $\mathbb C^{p,q}$ under $M$ is still an $r$-plane if and only if $ker(M)$ is negative semi-definite with respect to $H_{p,q}$. In such case, $M$ gives rise to a holomorphic map from each $D^r_{p,q}$ to the Grassmannian $Gr(r,\mathbb C^{p,q})$. It is clear that two matrices $M, M'\in M_{p+q}$ induce the same map on $D^r_{p,q}$ if $M=\lambda M'$ for some $\lambda\in\mathbb C^*$.
The question of interest is whether or not such a map is actually a \textit{self map} of $D^r_{p,q}$.

\begin{definition}
A self map of  $D^r_{p,q}$ is called a \textit{linear self map}, if it is given by a matrix $M\in M_{p+q}$ in the way described above. If there is no danger of confusion, we will also denote any such self map by the same symbol $M$. Conversely, any matrix $M\in M_{p+q}$ inducing a given linear self map of $D^r_{p,q}$ is called a \textit{matrix representation} of the linear self map.
\end{definition}

Let $M\in M_{p+q}$ be a matrix representation of a linear self map of some $D^r_{p,q}$. Then, we must have rank$(M)\geq p$ since otherwise $\dim_{\mathbb C}(ker(M))>q$ and $ker(M)$ would not be negative semi-definite in $\mathbb C^{p,q}$, contradicting the fact that $M$ induces a linear map on $D^r_{p,q}$. We 
make the following definition in relation to this.

\begin{definition}\label{minimality}
A linear self map of $D^r_{p,q}$ is called \textit{minimal} if it is given by a matrix $M\in M_{p+q}$ with rank$(M)=p$. Otherwise, we say that the linear self map is \textit{non-minimal}. 
\end{definition}

\noindent\textbf{Remark.} For the unit balls $D_{1,q}$, a non-minimal linear self map is simply a non-constant linear self map.
\newline


If $M\in M_n$ satisfies Inequality~\eqref{expansioneq}, then it follows immediately that the image of any positive definite $r$-plane is again a positive definite $r$-plane and thus $M$ induces a linear self map on each $D^r_{p,q}$. 
We are now going to show that conversely any \textit{non-minimal} linear self map of $D^r_{p,q}$ can be represented by a matrix satisfying Inequality~(\ref{expansioneq}). For this purpose, we will use some terminologies and results by Cowen-MacCluer~\cite{CM}. Let $(V,[\cdot,\cdot])$ be a finite dimensional complex vector space equipped with an indefinite Hermitian form $[\cdot,\cdot]$. Following~\cite{CM}, we will say that a linear map $T$ of $V$ into itself is an \textit{expansion} if $[Tv,Tv]\geq [v,v]$ for all $v\in V$, and an \textit{isometry} if $[Tv,Tv]=[v,v]$ for all $v\in V$. In particular, if we identify the linear maps of $\mathbb C^{p,q}$ into itself with their matrix representations with respect to the standard basis, then $M\in M_{p+q}$ is an expansion of $\mathbb C^{p,q}$ if and only if it satisfies the inequality~(\ref{expansioneq}) and is an isometry if $M$ makes the equality hold. We are going to show that the non-minimal linear self maps of $D^r_{p,q}$ are precisely those given by the expansions of $\mathbb C^{p,q}$, and the surjective linear self maps of  $D^r_{p,q}$ are given by the isometries of $\mathbb C^{p,q}$.

There following lemma can be found in  \cite{CM} but we rewrite it (reversing the signs) in a way more suitable for our purpose.


\begin{lem}[\cite{CM}] \label{expansion}
Suppose $[\cdot,\cdot]_1$ and  $[\cdot,\cdot]_2$ are indefinite  Hermitian forms on the complex vector space $V$ such that  $[x,x]_1=0$ implies $[x,x]_2 \geq 0$. Then,
$$\lambda:=-\inf_{[y,y]_1=-1}[y,y]_2 < \infty $$
and $$[x,x]_2\geq \lambda[x,x]_1$$
for all $x\in V$.
\end{lem}



We are now in a position to show that the non-minimal linear self maps of $D^r_{p,q}$ are all given by the matrices satisfying Inequality~(\ref{expansioneq}). (The case for the unit ball was obtained by Cowen-MacCluer~\cite{CM} and part of our proof is taken from there.)

\begin{thm}\label{expansion matrix}
Every non-minimal linear self map of $D^r_{p,q}$ can be represented by a matrix satisfying Inequality~(\ref{expansioneq}).
\end{thm}

\begin{proof}
Let $M\in M_{p+q}$ be a matrix such that it induces linear self map (also denoted by $M$) on some $D^r_{p,q}$. For $x,y\in\mathbb C^{p,q}$, let $[x,y]_1=y^H H_{p,q} x$ and let $[x,y]_2=(My)^HH_{p,q} Mx$. The hypothesis that $M$ maps
 $D^r_{p,q}$ into $D^r_{p,q}$ means that whenever $[x,x]_1>0$, we have $[x,x]_2>0$. By continuity, we get that if $[x,x]_1 \geq 0$, then $[x,x]_2 \geq 0$.
  Thus, the hypotheses of Lemma \ref{expansion} are satisfied. So we only need to show that $\la > 0$. Then $\la^{-1/2}M$ is an expansion in $\mathbb C^{p,q}$.

To see that $\lambda >0$, we let the range of $M$ be $R(M):=\left\{y\in\mathbb C^{p,q}: y=Mz \textrm{\,\,for some\,\,} z\right\}$. Now if $M$ is non-minimal, then
$\dim_{\mathbb C}(R(M))\geq p+1$ (Definition~\ref{minimality}). Thus, $R(M)$ must contain a negative vector $y$ and any preimage $z$ of $y$ must also be a negative vector since we have already seen in the previous paragraph that $[z,z]_1 \geq 0$ would imply that $[y,y]_1=[z,z]_2 \geq 0$.
Therefore, we can find $z\in \mathbb{C}^{p,q}$ such that $[z,z]_1<0$ as well as $[z,z]_2<0$ and hence $\la > 0$. 

\end{proof}

\begin{thm}\label{isometrythm}
Every surjective linear self map of $D^r_{p,q}$ can be represented by a matrix satisfying the equality in~(\ref{expansioneq}). In particular, every surjective linear self map of $D^r_{p,q}$ is an automorphism.
\end{thm}
\begin{proof}
If a linear self map $M$ of a given $D^r_{p,q}$ is surjective, then $M$ is surjective as a linear map on $\mathbb C^{p,q}$ since its image contains all the positive vectors (which constitute an open set in $\mathbb C^{p,q}$). The inverse linear map $M^{-1}$ also maps positive vectors to positive vectors since each of the 1-planes generated by positive vectors can be regarded as the intersection of a set of positive $r$-planes and $M$ is surjective (as a self map of $D^r_{p,q}$). Hence, we see that $M^{-1}$ is also a surjective linear self map of $D^r_{p,q}$ and is thus an expansion. Therefore, there are non-zero scalars $\alpha$ and $\beta$ such that $\alpha M$ and $\beta M^{-1}$ are expansions. Thus, for every $z\in\mathbb C^{p,q}$, if we write $\|z\|^2_{p,q}:=z^H H_{p,q} z$, then
$$
	\|\beta z\|^2_{p,q}\geq \|Mz\|^2_{p,q}\geq \|\alpha^{-1} z\|^2_{p,q}
$$
and hence
$$
	|\alpha\beta|^2\|z\|^2_{p,q}\geq\|\alpha Mz\|^2_{p,q}\geq\|z\|^2_{p,q}.
$$

Since the inequality is true for both positive vectors and negative vectors, we deduce that $|\alpha\beta|=1$ and $\|\alpha Mz\|^2_{p,q}=\|z\|^2_{p,q}$ for every $z$. That is, $\alpha M$ is an isometry.
\end{proof}

A linear self map $M$ of $D^r_{p,q}$ originally comes from a linear map $\tilde M$ defined on the ambient Grassmannian $Gr(r,\mathbb C^{p,q})$. If $M$ is not surjective, then there is a set of indeterminacy $Z\subset Gr(r,\mathbb C^{p,q})$ on which $\tilde M$ is not defined. The set $Z$  is outside $D^r_{p,q}$ and consists of the points corresponding to the $r$-planes that intersect the kernel of a matrix representation of $\tilde M$. A priori, $Z$ can intersect the boundary $\partial D^r_{p,q}$ and obstructs the extension of $M$ across $\partial D^r_{p,q}$, but we are now going to show that this does not happen for non-minimal linear self maps. 

\begin{thm}\label{extension thm}
Every non-minimal linear self map of $D^r_{p,q}$ extends holomorphically to an open neighborhood of the closure $\overline{D^r_{p,q}}:=D^r_{p,q}\cup\partial D^r_{p,q}$.
\end{thm}
\begin{proof}
Let $M$ be a non-minimal linear self map of $D^r_{p,q}$ and denote also by $M$ a matrix representation of it which satisfies the inequality $M^HH_{p,q}M-H_{p,q}\geq 0$. As mentioned at the beginning of this section, $ker(M)$ must be negative semi-definite with respect to $H_{p,q}$. We are now going to show that $ker(M)$ does not contain any non-zero null vector if $M$ is non-minimal. Suppose on the contrary there is a non-zero null vector $\eta\in ker(M)$. Let $v\in\mathbb C^{p,q}$ and write $\|v\|_{p,q}^2=v^HH_{p,q}v$. Then for any $r\in\mathbb R$, we have $M(v+r\eta)=Mv$ and
$$
\|Mv\|_{p,q}^2=\|M(v+r\eta)\|_{p,q}^2\geq \|v+r\eta\|_{p,q}^2=\|v\|_{p,q}^2+r(\eta^HH_{p,q}v+v^HH_{p,q}\eta).
$$

Now if $v$ is chosen such that $\textrm{Re}(v^HH_{p,q}\eta)\neq 0$, then the above inequality cannot hold for every $r\in\mathbb R$ and hence we get a contradiction. Consequently, $ker(M)$ does not contain any non-zero null vector and therefore the image of any positive semi-definite $r$-plane under $M$ is still an $r$-plane and the set of indeterminacy $Z\subset Gr(r,\mathbb C^{p,q})$ of $M$ (as a linear map on $Gr(r,\mathbb C^{p,q}))$ is disjoint from $\overline{D^r_{p,q}}$. Since both $Z$ and $\overline{D^r_{p,q}}$ are closed in $Gr(r,\mathbb C^{p,q})$ (and hence compact), there is an open neighborhood of $\overline{D^r_{p,q}}$ disjoint from $Z$ and now the theorem follows.
\end{proof}

\noindent\textbf{Remark.} The non-minimality is indeed necessary to guarantee the extension across the entire boundary. See Example~\ref{no extension} for a minimal linear self map which does not extend across some boundary point.


\section{Automorphisms on generalized balls}\label{automorphisms}

In this section, we are going to study in detail the automorphisms on the generalized balls $D_{p,q}$, regarding their fixed point sets and also their normal form. We begin by determining the automorphism group of $D_{p,q}$.

\begin{thm}
Every automorphism of $D_{p,q}$ is a linear self map and thus extends to an automorphism of the ambient projective space $\mathbb P^{p+q-1}$. The automorphism group $\textrm{Aut}(D_{p,q})$ of $D_{p,q}$ is isomorphic to $PU(p,q)$, the projectivization of the indefinite unitary group $U(p,q)$. In particular, every automorphism of $D_{p,q}$ can be represented by a matrix in $U(p,q)$.
\end{thm}
\begin{proof}
The statements are well-known for the complex unit balls $D_{1,q}$ and also for the complements of the complex unit balls $D_{p,1}$. Suppose now $p,q\geq 2$. Then, it has been shown by Baouendi-Huang~\cite{BH} and (for a more geometric proof, see Ng~\cite{ng1}) that every automorphism of $D_{p,q}$ is necessarily a linear map. Now by Theorem~\ref{isometrythm} here (or Lemma 2.13 in~\cite{ng1}), we see that every automorphism can be represented by a matrix in $U(p,q)$. Since two elements in $U(p,q)$ represent the same automorphism of $D_{p,q}$ if and only if they are scalar multiples of each other, it follows now that $\textrm{Aut}(D_{p,q})\cong PU(p,q)$.
\end{proof}

\begin{cor}\label{extension cor}
The action of $\textrm{Aut}(D_{p,q})$ extends real-analytically to $\partial D_{p,q}$.
\end{cor}



The following version of Witt's theorem is very useful in studying the various transitivities of $\textrm{Aut}(D_{p,q})$.

\begin{lem}[Witt \cite{Witt}] \label{cor_Witt theorem}
Let $X$ be a complex vector space equipped with a non-degenerate Hermitian form and $Y \subset  X$ be any complex vector subspace. Then any isometric embedding $f : Y \rightarrow X$ extends to an isometry $F$ of $X$.
\end{lem}

\begin{thm}\label{doubly transitive} For $u\in \mathbb C^{p,q}$, let $[u]$ be its projectivization in  $\mathbb P^{p+q-1}$.
\begin{enumerate}
\item
 $\text{Aut}(D_{p,q})$ is transitive on $D_{p,q}$ and also on $\partial D_{p,q}$.
\item
Let $[v_1]$, $[v_2]$, $[w_1]$, $[w_2]\in \partial D_{p,q}$, where $[v_1]\neq [v_2]$ and $[w_1]\neq [w_2]$. Then, there exists $M\in \textrm{Aut}(D_{p,q})$ such that $M([v_j])=[w_j]$ for $j=1,2$ if and only if there exists non-zero $\alpha\in\mathbb C$ such that $v_1^HH_{p,q}v_2=\alpha w_1^HH_{p,q}w_2$.
\end{enumerate}
\end{thm}
\begin{proof}
{ 1. Let $[u_1]$, $[u_2]\in D_{p,q}$. Then we choose some $k>0$ such that the linear map $f:\mathbb Cu_1\rightarrow \mathbb C^{p,q}$, defined by $u_1\mapsto ku_2$ is an isometric embedding. By Lemma~\ref{cor_Witt theorem} $f$ extends to an isometry of $\mathbb C^{p,q}$ and therefore there is an automorphism of $D_{p,q}$ mapping $[u_1]$ to $[u_2]$. Similarly, for any two points $[v_1]$, $[v_2] \in \partial D_{p,q}$. The map
$i\colon \mathbb C v_1\rightarrow \mathbb C^{p,q}$ defined by $v_1\mapsto v_2$ is an isometric embedding and hence there exists an automorphism of $D_{p,q}$ such mapping $[v_1]$ to $[v_2]$.
}

2. Suppose that there exists $\alpha\in\mathbb C^*$ such that $v_1^HH_{p,q}v_2=\alpha w_1^HH_{p,q}w_2$. Let $\phi : Y \rightarrow \CC^{p,q}$, where $Y=\textrm{Span}\{v_1, v_2\}$ be the linear embedding defined by $\phi(v_1)=w_1$ and $\phi(v_2)=w_2$. By replacing $w_1$ by some scalar multiple, we can make $v_1^HH_{p,q}v_2=w_1^HH_{p,q}w_2$. Then $\phi$ is an isometric embedding and hence $\phi$ extends to an isometry of $\mathbb C^{p,q}$ by Lemma \ref{cor_Witt theorem} and the desired result follows. The converse is trivial.
\end{proof}

\noindent\textbf{Remark.}
Theorem~\ref{doubly transitive} is a generalization of the double transitivity of the automorphism groups of the complex unit balls on their boundaries. This is because for $[v_1], [v_2]\in\partial D_{1,q}$ with $[v_1]\neq [v_2]$, we always have $v_1^HH_{p,q}v_2\neq 0$ since otherwise we would get a two dimensional isotropic subspace in $\mathbb C^{1,q}$.

\subsection{Fixed points on $D_{p,q}$ and $\partial D_{p,q}$}

Let $A\in U(p,q)$ and denote also by $A$ the corresponding automorphism of $D_{p,q}$.
It follows directly from the definition of the matrix representation of a linear self map that the fixed points of $A$ (as an automorphism on $D_{p,q}$) correspond precisely to the one-dimensional eigenspaces (or projectivized eigenvectors) of $A$ associated to the non-zero eigenvalues. The following simple observation regarding eigenvalues and eigenvectors of matrices in $U(p,q)$ will be very useful in studying the fixed points of the automorphisms of $D_{p,q}$.

\begin{lem}\label{ortho lemma}
Let $A\in U(p,q)$. If $\lambda_1$, $\lambda_2$ are eigenvalues of $A$ and $v_1$, $v_2$ are two eigenvectors associated to them respectively, then either $v_1$ and $v_2$ are orthogonal with respect to $H_{p,q}$ or $\overline{\lambda_2}\lambda_1=1$.
\end{lem}
\begin{proof}
The result follows from
$
v_2^H H_{p,q} v_1= v_2^HA^H H_{p,q} Av_1=\left(\ov{\lambda}_2\lambda_1\right)\,v^H_2 H_{p,q} v_1.
$
\end{proof}

Let $\ov{D_{p,q}}:=D_{p,q}\cup\partial D_{p,q}$ be the closure of $D_{p,q}$. Since the closed complex unit balls $\overline{\mathbb B^q}\cong\overline{D_{1,q}}$ is a convex compact set in $\mathbb C^q\cong \mathbb R^{2n}$, it follows from Corollary~\ref{extension cor} and Brouwer's fixed point theorem that every element in $Aut(D_{1,q})$ has a fixed point in $\overline{D_{1,q}}$. 
When $p\geq 2$, any $D_{p,q}$ cannot be embedded as a convex compact set in some Euclidean space since it contains positive dimensional projective subspaces (see~\cite{ng1}). Nevertheless, with a bit of ``detour", we will still be able to use  Brouwer's fixed point theorem to get a fixed point in the closure for \textit{any} linear self map of $D_{p,q}$. The proof will be given in Section~\ref{linear maps section}. Hence, we have

\begin{thm}\label{existence of fixed points}
Every element in $\textrm{Aut}(D_{p,q})$ has a fixed point on $\ov{D_{p,q}}:=D_{p,q}\cup\partial D_{p,q}$.
\end{thm}
\begin{proof}
Follows directly from Theorem~\ref{general existence of fixed points}.
\end{proof}


We now recall some elementary linear algebra. Let $M\in M_n$ and $\lambda$ be an eigenvalue of $M$. For some $r\in \mathbb N$ a vector $v\in\mathbb C^n$ is called a generalized eigenvector of rank $r$ of $M$ associated to the eigenvalue $\lambda$ if $(M-\lambda I)^rv=0$ but $(M-\lambda I)^{r-1}v\neq 0$.
It turns out that both the absolute values of the eigenvalues and the existence of generalized eigenvectors of higher rank give information about the fixed-point set of the linear self map on a generalized ball represented by $M$.

As before, for a vector $v\in\mathbb C^{p,q}$, we will denote by $[v]\in\mathbb P^{p+q-1}$ its projectivization. Similarly, for any complex vector subspace $W\subset\mathbb C^{p,q}$, we denote its projectivization by $[W]$.

\begin{prop}\label{fixed point}
Let $A\in \text{Aut}(D_{p,q})$ and choose a matrix representation in $U(p,q)$ and denote it also by $A$. Let $\lambda$ be an eigenvalue of $A$ and $v$ be an associated generalized eigenvector of rank $r$.
\begin{enumerate}\label{projective subspace on the boundary}
\item If $|\lambda|\neq 1$, then $[(A-\lambda I)^{r-1}v]$ is a fixed point of $A$ on $\partial D_{p,q}$. Furthermore, if $r\geq 2$, then there is an $(r-1)$-dimensional projective linear subspace $[W]$ in $\partial D_{p,q}$
invariant under $A$ and $[(A-\lambda I)^{r-1}v]\in [W]$ is a unique fixed point of $A$ in $[W]$.
\item If $|\lambda|=1$ and $r\geq 2$,
then $[(A-\lambda I)^{r-1}v]$ is a fixed point of $A$ on $\partial D_{p,q}$.
Furthermore, there is an $\left(\left[ \frac{r}{2} \right]-1\right) $-dimensional projective subspace
$[W]$ in $\partial D_{p,q}$ invariant under $A$ and $[(A-\lambda I)^{r-1}v]$ is a unique fixed point of $A$ in $[W]$.
\end{enumerate}
\end{prop}
\begin{proof}
1. Let $A\in U(p,q)$, $\lambda$ be an eigenvalue of $A$ and $v$ be an associated generalized eigenvector of rank $r$. Let $v_r:=v$ and define inductively $v_{j-1}:=(A-\lambda I)v_j$ for $j\in\{r,r-1,\ldots, 2\}$. In particular $v_1$ is an eigenvector of $A$ associated to $\lambda$.
\begin{eqnarray*}
Av_1&=&\lambda v_1, \\
Av_2 &=& v_1 + \lambda v_2,\\
&&\vdots\\
A v_r &=&v_{r-1} + \lambda v_r.
\end{eqnarray*}
Suppose that $|\lambda|\neq 1$.
We claim that
\begin{equation}
v_i^HH_{p,q} v_j=0 \text{ for all } 1\leq i,j\leq r.
\end{equation}
We will prove it using induction.
Because of $A^H H_{p,q} A = H_{p,q}$, we have
$v_1^H H_{p,q}v_1 =0$
and this implies that $[v_1]\in \partial D_{p,q}$.
Suppose that  $r\geq 2$ and $v_1^H H_{p,q}v_{j'}=0$ for every $j'<j\leq r$. Then
\begin{eqnarray*}
v_1^H H_{p,q}v_{j} = (Av_1)^H H_{p,q}(Av_j)
= \ov\lambda v_1^H H_{p,q}(v_{j-1}+\lambda v_j)
=|\lambda|^2 v_1^HH_{p,q}v_j.
\end{eqnarray*}
Hence $v_1^HH_{p,q}v_j=0$ and as a consequence, 
\begin{equation}
v_1^HH_{p,q}v_j=0 \quad\text{ and }\quad v_j^HH_{p,q}v_1=0
\quad\text{ for all } j \,\text{ with } \,1\leq j\leq r.
\end{equation}
Now fix $i\geq2$ and $j\geq 2$.
For the induction, assume that $v_{i'}^HH_{p,q} v_{j'}=0$ for all $i'<i$ or $j'<j$.
Since
\begin{eqnarray*}
v_i^HH_{p,q} v_j &=& (v_{i-1}+\lambda v_i)^H H_{p,q} (v_{j-1}+\lambda v_j)\\
&=& v_{i-1}^H H_{p,q}v_{j-1} + \lambda v_{i-1}^HH_{p,q} v_j
+ \ov\lambda v_i^HH_{p,q} v_{j-1} + |\lambda|^2 v_i^HH_{p,q} v_j,
\end{eqnarray*}
we can obtain $v_i^HH_{p,q} v_j=0$ and we obtain the claim.
Define the subspace $W$ in $\CC^{p,q}$ spanned by $v_1,\ldots, v_r$.
Then $[W]$ is a projective subspace in $\mathbb P^{p+q-1}$ contained in $\partial D_{p,q}$ which is invariant with respect to the action of $A$ and it has a unique fixed point $[v_1]$ since $v_1$ is the unique eigenvector in $W$ (up to scalar multiplication).

2. Consider the case $|\lambda|=1$.
By replacing $A$ with $\frac{1}{\lambda}A$,
we may assume that $\lambda=1$ without any loss of generality.
For $j$ with $1\leq j\leq r-1$ we have
$$v_1^H H_{p,q}v_{j+1} = (Av_1)^H H_{p,q} (Av_{j+1}) = v_1^H H_{p,q} (v_{j+1}+v_j).$$
This implies that
\begin{equation}\label{1j}
v_1^HH_{p,q}v_j=0 \,\text{ and }\, v_j^HH_{p,q}v_1=0 \,\text{ for }\, j\, \text{ with }\, 1\leq j\leq r-1.
\end{equation}
Since for $m$ with $2\leq m\leq r-1$ we have
$$
v_2^H H_{p,q}v_m = (Av_2)^H H_{p,q}(Av_m) = (v_1+v_2)^HH_{p,q}(v_{m-1}+v_m) = v_2^HH_{p,q}(v_{m-1}+v_m)
$$
by \eqref{1j}, one obtains
\begin{equation}
v_2^HH_{p,q}v_m =0 \,\,\text{ and }\,\, v_m^HH_{p,q}v_2 =0 \,\,\text{ for all }\,\,
m \,\,\text{ with }\,\,1\leq m\leq r- 2.
\end{equation}
By repeating this process, we obtain that for a fixed $j$ with $1\leq j\leq r-1$,
\begin{equation}
v_m^H H_{p,q} v_j=0 \,\,\text{ for all }\,\, m \,\,\text{ with } \,\,1\leq m\leq r-j.
\end{equation}
As a result
\begin{equation}
v_m^H H_{p,q}v_j=0 \,\,\text{ for all }\,\, m,\,j \,\,\text{ with } \,\,1\leq m,j\leq \left[ \frac{r}{2} \right].
\end{equation}
Define the complex vector subspace $W$ of $\CC^{p,q}$ spanned by $v_1,\ldots, v_{\left[ \frac{r}{2} \right]}$.
Then $[W]$ is a projective subspace in $\mathbb P^{p+q-1}$ contained in $\partial D_{p,q}$ which is invariant with respect to the action of $A$ and it contains a unique fixed point $[v_1]$.
\end{proof}


\begin{cor}
Let $A\in Aut(D_{p,q})$ and choose a matrix representation in $U(p,q)$ and denote it also by $A$. Then, in any Jordan canonical form of $A$, every higher-rank Jordan block or rank-one Jordan block with a non-unimodular eigenvalue corresponds to a fixed point on $\partial D_{p,q}$.
\end{cor}

\noindent\textbf{Remark.} There exist indeed non-diagonalizable elements in $U(p,q)$. We refer the reader to Example~\ref{nondiagonalizable}.
\newline

Through the generalized Cayley transform one can map the complex unit balls $D_{1,q}$ biholomorphically onto some Siegel domains of the second kind. Thus, the dilations on the Siegel domains give elements in $\textrm{Aut}(D_{1,q})$ which do not have fixed points in $D_{1,q}$ and hence there exist fixed points on the boundary by Brouwer's fixed point theorem.  On the other hand, Hayden-Suffridge~\cite{hs} showed that if an automorphism of a complex unit ball has more than two fixed points on the boundary then it must have fixed points in the interior (in fact, there is at least an affine line on which every point is fixed). For $p,q\geq 2$, the generalized balls $D_{p,q}$ contain projective linear subspaces in their boundaries and this leads to a lot of differences between the complex unit balls and other generalized balls when studying their holomorphic mappings. Example~\ref{hs example} in Section~\ref{examples} will show that the result of Hayden-Suffridge cannot be generalized to $D_{p,q}$ by simply increasing the number of fixed points on the boundary in relation to the existence of these projective subspaces in the boundary. Before getting a suitable generalization of Hayden-Suffridge (in Corollary~\ref{hs generalization}), we observe the following general behavior relating the fixed points and the projective lines on which every point is fixed.

\begin{thm} \label{fixed point number thm}
Let $A\in \text{Aut}(D_{p,q})$ and choose a matrix representation in $U(p,q)$ and denote it also by $A$.
\begin{enumerate}
\item
If $A$ has at least $p+1$ fixed points in $D_{p,q}$, then
there exists a projective line intersecting $D_{p,q}$ on which every point is fixed by $A$. In particular, if an element in $\textrm{Aut}(D_{p,q})$ has only a finite number of fixed points in $D_{p,q}$, then it can have at most $p$ fixed points in $D_{p,q}$.
\item
If A has at least $2p+1$ fixed points in $\partial D_{p,q}$, then there exists a projective line on which every point is fixed by $A$.
\end{enumerate}
\end{thm}

\begin{proof}

(1) Suppose that $A$ has $p+1$ fixed points $\{[v_1],\ldots,[v_{p+1}]\}$ in $D_{p,q}$ associated to $p+1$ distinct eigenvalues $\{\lambda_1,\ldots,\lambda_{p+1}\}$.
By Lemma~\ref{ortho lemma}, we have either $\overline{\lambda_i}\lambda_j=1$ or
$v_i^HH_{p,q}v_j=0$.
However, we must have $|\lambda_j|=1$ for all $j$ by (1) in Proposition \ref{fixed point}, and all $\lambda_j$ are distinct, we see that for $i\neq j$, $\overline{\lambda_i}\lambda_j=1$ is impossible. This implies that
$v_i^HH_{p,q}v_j=0$ whenever $i\neq j$. Then $\{v_1,\ldots, v_{p+1}\}$ span a positive definite $(p+1)$-dimensional subspace in $\mathbb C^{p,q}$, which is a contradiction.  Thus, at least two elements in $\{v_1,\ldots, v_{p+1}\}$ are associated to the same eigenvalue and they span a 2-dimensional eigenspace and this gives a projective line on which every point is fixed by $A$.

(2) Suppose that $A$ has $2p+1$ fixed points in $\partial D_{p,q}$ associated to $2p+1$ distinct eigenvalues. Pick any one fixed point and denote it by $[v_1]$, with the associated eigenvalue denoted by $\lambda_1$. In the remaining $2p$ fixed points, there is at most one of them is associated to the eigenvalue $(\overline{\lambda_1})^{-1}$. Thus, there are at least $2p-1$ of them whose corresponding projectivized eigenvectors are orthogonal to $[v_1]$ by Lemma~\ref{ortho lemma}. Pick any one such fixed point and denote it by $[v_2]$, with the associated eigenvalue denoted by $\lambda_2$. By repeating this procedure we see that we can choose $p+1$ fixed points whose corresponding projectivized eigenvectors are pairwise orthogonal and thus we get a $(p+1)$-dimensional isotropic subspace in $\mathbb C^{p,q}$, which is a contradiction. Thus, at least two fixed points are associated to the same eigenvalue and we again get a projective line on which every point is fixed by $A$, as in $(1)$.
\end{proof}

\noindent\textbf{Remark.} The numbers $p+1$ and $2p+1$ in (1) and (2) are sharp. It is illustrated by  Example~\ref{sharp example 3} and Example~\ref{sharp example 2} in Section~\ref{examples}.

\begin{cor}\label{at most}
Let $A\in \text{Aut}(D_{p,q})$ and
suppose that there does not exist any projective line on which every point is fixed by $A$. Then $A$ has at most $p$, $q$ and $\min\{2p,2q\}$  fixed points in $D_{p,q}$, $\mathbb P^{p+q-1}\setminus \overline{ D_{p,q}}$ and $\partial D_{p,q}$ respectively.
\end{cor}

We can now give the following generalization of Hayden-Suffridge's result for the generalized balls. (This also gives an alternative proof for their result for the unit balls of finite dimension.)

\begin{cor}\label{hs generalization}
Let $A\in\textrm{Aut}(D_{p,q})$ be an automorphism such that there does not exist any projective line in $\partial D_{p,q}$ on which every point is fixed by $A$. If $A$ has at least $2p+1$ fixed points on $\partial D_{p,q}$, then $A$ must have fixed points in $D_{p,q}$.
\end{cor}

\begin{proof}
By (2) in Theorem~\ref{fixed point number thm}, we get a projective line $L$ on which every point is fixed by $A$. Furthermore, from its proof we know that $L$ intersects $\partial D_{p,q}$ at at least two points. This means we have two linearly independent null vectors in the 2-dimensional subspace of $\mathbb C^{p,q}$ whose projectivization is $L$. Thus, this 2-dimensional subspace is either isotropic or contains positive vectors and this in turn means that either $L$ is completely contained in $\partial D_{p,q}$ or it intersects $D_{p,q}$. The desired result now follows.
\end{proof}

\subsection{Normal forms of automorphisms on $D_{p,q}$}

In this section we will find a normal form elements in $U(p,q)$, which are matrices representing the automorphisms of $D_{p,q}$. It is a generalization of a result given in
\cite[Proposition 5]{CM} for the unit ball case.

For a point $v\in \mathbb C^{p,q}$, write $v=(v', v'')$ with
$v'\in \mathbb C^p$ and $v''\in \mathbb C^q$. Also write $\|v\|^2_{p,q}:=v^HH_{p,q}v=|v'|^2_p-|v''|^2_q$, where $|\cdot|^2_p$ and $|\cdot|^2_q$ denote the Euclidean norms.
If we naturally identify $U(p)\times U(q)$ as a subgroup of $U(p,q)$ (as block diagonal matrices), then there exists $U \in U(p)\times U(q)$ such that
$Uv = (|v'|_p,0,\ldots, 0, |v''|_q)$.

Now suppose $v$ is a unit vector, i.e. $\|v\|_{p,q}^2 =1$. Define
$$M = \left( \begin{array}{ccccc}
|v'|_p & &&&-|v''|_q \\
 &1 &&& \\
& & \ddots&&\\
 & & &1 & \\
|v''|_q & & &&-|v'|_p \\
\end{array}\right),
$$
where other entries are zero. Then, $M\in U(p,q)$, $M^{-1}=M$ and we have $MUv=(1,0,\ldots, 0)^t$.
Hence, we also have $v = U^{-1}M(1,0,\ldots, 0)^t$.

Now let $A\in U(p,q)$ and let $v\in\mathbb C^{p,q}$ be the vector such that $Av = (1,0,\ldots, 0)^t$. In particular, $\|v\|^2_{p,q}=1$.
Then from above we see that there exist $V_1\in U(p)\times U(q)$ and $M_1\in U(p,q)$ such that $AV_1 M_1(1,0,\ldots, 0)^t = (1,0,\ldots, 0)^t$.
This implies that $AV_1M_1$ belongs to the isotropy subgroup at $(1,0,\ldots,0)^t$, which is $\{1\}\times U(p-1,q)$ (as block diagonal matrices in $U(p,q)$).
Since
$\{1\}\times U(p-1,q)$ preserves the subspace $\{v\in\mathbb C^{p,q}: v=(0,*,\ldots,*)\}$, which
is isometric to $\mathbb C^{p-1,q}$, we can repeat the argument for $p\geq 2$ and find
$V_2\in \{1\}\times U(p-1)\times U(q)\subset U(p)\times U(q)$ and $M_2\in\{1\}\times U(p-1,q)\subset U(p,q)$ 
such that $AV_1M_1V_2M_2$ fixes both $(1,0,\ldots,0)^t$ and $(0,1,0,\ldots,0)^t$, where $M_2$ is of the following form for some $a,b\geq 0$ such that $a^2-b^2=1$,
$$\left( \begin{array}{c|ccccccc}
 1& &&&&  &\\\hline
 & a&&&&  &-b\\
 & &1 &&& & \\
& && \ddots& && \\
 & & &&1& & \\
 &b && &&&-a \\
\end{array}\right).
$$

Hence, we see that $AV_1M_1V_2M_2\in \{I_2\}\times U(p-2,q)$.
By repeating this process, we can find $V_j\in U(p)\times U(q)$, $M_j\in U(p,q)$, $1\leq j\leq p$, so that $A V_1M_1\cdots V_pM_{p}\in \{I_p\}\times U(q)$.
We call the automorphisms associated to the matrices of the form as $M_j$, $1\leq j\leq p$, the {\it non-isotropic dilations} of $D_{p,q}$.
In conclusion, we have obtained the following:
\begin{thm}\label{normal form}
The subgroup $U(p)\times U(q)$ and the non-isotropic dilations generate the full automorphism group of  $D_{p,q}$. Furthermore, every element in $Aut(D_{p,q})$ can be written in the form $U_{p+1}M_pU_p\cdots M_1U_1$, where $U_{1},\ldots,U_{p+1}$ are automorphisms fixing the subspace $\{[z]\in D_{p,q}: [z]=[z_1,\ldots,z_p,0,\ldots,0]\}$ and $M_1,\ldots,M_p$ are non-isotropic dilations of $D_{p,q}$.
\end{thm}


\section{Linear self maps on generalized balls}\label{linear maps section}
\subsection{Fixed points of general linear self maps}

For any linear self map of a unit ball $D_{1,q}$, one can apply Brouwer's fixed point theorem to conclude that there is at least one fixed point in the closure of the ball. However, the argument cannot be directly carried over to other generalized balls since they cannot be realized as bounded convex domains in the Euclidean space (see~\cite{gn}). Nevertheless, we are going to prove that the same fixed point theorem still holds for the generalized balls. We first recall that the fixed points of a linear self map on a generalized ball correspond to the one-dimensional eigenspaces associated to the non-zero eigenvalues of any given matrix representation of the linear self map.

\begin{thm}\label{general existence of fixed points}
Every linear self map of $D_{p,q}$ has at least one fixed point in $\overline{D_{p,q}}:=D_{p,q}\cup\partial D_{p,q}$.
\end{thm}
\begin{proof}
Let $F$ be a linear self map of $D_{p,q}$ and denote by $A$ one of its matrix representations. Then, $A$ maps the positive $p$-planes in $\mathbb C^{p,q}$ onto positive $p$-planes since it maps positive vectors to positive vectors (for $F$ to be well defined on $D_{p,q}$). In other words, $A$ also induces a linear self map $\tilde F$ of $D^p_{p,q}$, in which the latter is a classical bounded symmetric domain of type-I embedded in $Gr(p,\mathbb C^{p,q})$ and if we use the inhomogeneous coordinates in a standard Euclidean chart $\mathbb C^{pq}\subset Gr(p,\mathbb C^{p,q})$, then $D^p_{p,q}\Subset\mathbb C^{pq}$ is just the standard Harish-Chandra realization of the bounded symmetric domain (see~\cite{ng2}) and hence is a bounded convex domain~\cite{H}. Now in terms of these Euclidean coordinates, $\tilde F$ is a rational map and since $\tilde F(D^p_{p,q})\subset D^p_{p,q}$, and the latter is a bounded domain in $\mathbb C^{pq}$, we deduce that the rational map $\tilde F$ is also well defined (holomorphic) in a neighborhood of the closure $D^p_{p,q}\cup\partial D^p_{p,q}$. We can now apply Brouwer's fixed point theorem to get a fixed point $x_0\in D^p_{p,q}\cup\partial D^p_{p,q}$ of $\tilde F$. Since $x_0$ corresponds to a positive semi-definite $p$-plane $E_0\subset\mathbb C^{p,q}$, and from how $\tilde F$ is constructed from $A$, we see that $A(E_0)=E_0$. Then $A$ has at least one eigenvector in $E_0$ associated to a non-zero eigenvalue. Such an eigenvector is either a positive vector or a null vector and it gives us a fixed point in $\overline{D_{p,q}}$.
\end{proof}

The following lemma will be useful for further analyzing the fixed points of a linear self map on a generalized ball.

\begin{lem}\label{semipositive}
If $A$ is a matrix satisfying $A^H H_{p,q}A-H_{p,q}\geq 0$ and $v\in\mathbb C^{p,q}$ is such that $\|Av\|_{p,q}^2=\|v\|_{p,q}^2$, then $A(v^\perp)\subset v^\perp$, where $v^\perp$ is the orthogonal complement of $v$ with respect to $H_{p,q}$.
\end{lem}

\begin{proof}
Since $A^H H_{p,q}A-H_{p,q}$ is a positive semi-definite Hermitian matrix, there exists a unitary matrix $P$ such that $P^{H}(A^H H_{p,q}A-H_{p,q})P=\text{diag}(\lambda_1, \lambda_2, \ldots,\lambda_{r},$ $0,\ldots,0)$, with $\lambda_{i}>0$ for all $i$.

As $\|Av\|_{p,q}^2=\|v\|_{p,q}^2$, we have $v^H (A^H H_{p,q}A-H_{p,q})v=0$. 
Let $v'=P^Hv=({x'}_1,x'_2,\ldots, x'_{p+q})$. Then $x'_1=x'_2=\cdots=x'_r=0$ since $v'^H \text{diag}(\lambda_1, \lambda_2, \ldots,\lambda_{r},0,\ldots,0)v'=v^H (A^H H_{p,q}A-H_{p,q})v=0$. It follows that $P^H(A^H H_{p,q}A-H_{p,q})PP^Hv=0$ and hence $(A^H H_{p,q}A-H_{p,q})v=0$. Then $u^H (A^H H_{p,q}A-H_{p,q})v=0$ for any $u\in \mathbb{C}^{p,q}$.
 Therefore $(Au)^HH_{p,q}Av=u^HH_{p,q}v$. So $A(v^\perp)\subset v^\perp$.
\end{proof}

The following generalization of the fixed point theorem of Hayden-Suffridge to linear fractional maps of the complex unit ball was accomplished by Bisi-Bracci~\cite{BB}: \textit{If a linear fractional map of the complex unit ball has more than
two fixed points on the boundary, then it has fixed points in the interior.} Just like what happens for automorphisms (as studied in Section~\ref{automorphisms}), we will have to assume that no projective line in the boundary is fixed everywhere by the linear self map before we can generalize this result to $D_{p,q}$.


We have proven in Theorem~\ref{extension thm} that any non-minimal linear self map of a generalized ball extends holomorphically across the boundary. For minimal linear self maps, we still have such an extension for a \textit{general} boundary point since the set of indeterminacy cannot contain the entire boundary. But even so, a minimal linear self map cannot have any boundary fixed point. In order to see this, let $F$ be a minimal linear self map of a generalized ball. The range $R$ of its matrix representation (also denoted by $F$) is of dimension $p$ (see Definition~\ref{minimality}). On the other hand, since $F$ maps positive vectors to positive vectors, for any positive $p$-plane $P\subset\mathbb C^{p,q}$, we must have $\dim_\mathbb C F(P)=p$ and hence $F(P)=R$, which implies that $R$ is a positive $p$-plane. Hence, the image of any null vector under $F$ is either a positive vector or the zero vector. Therefore, a minimal linear self map cannot have any fixed point on the boundary.

We will first establish our fixed point theorem for the complex unit ball and its exterior, i.e. for $D_{1,q}$ and $D_{p,1}$. In particular, this will give a new proof for the result of Bisi-Bracci~\cite{BB}.

\begin{thm}\label{ball}

Let $q\geq 1$ and $p\geq 2$. For $D_{1,q}$ (resp. $D_{p,1}$), any linear self map with at least three (resp. two ) boundary fixed points has a fixed point in the interior.

\end{thm}

\begin{proof}

We will first prove the theorem for $D_{1,q}$. Let $F$ be a linear self map of $D_{1,q}$ with at least three fixed points in $\partial D_{1,q}$. Let $A$ be a matrix representation of $F$ and by the hypotheses we can find three null vectors $v_1,v_2,v_3\in\mathbb C^{1,q}$ such that any two of them are not proportional and they are eigenvectors of $A$ associated to some non-zero eigenvalues.

First of all, when $q=1$, then $\dim_\mathbb C(\mathbb C^{1,1})=2$. But $v_1,v_2,v_3$ are pairwise non-proportional eigenvectors and so we deduce that $A$ can only be a non-zero multiple of the identity matrix and the desired result follows in this case.

Now we can assume that $q\geq 2$. Let $E=span_\mathbb C(v_1,v_2,v_3)$. Then $E$ is an invariant subspace of $A$. Moreover, the restriction $H_{1,q}|_E$ is non-degenerate since there is no isotropic subspace in $\mathbb C^{1,q}$ of dimension greater than one.

 If $\dim_\mathbb C(E)=2$, then the $H_{1,q}|_E$ must be of signature $(1,1)$ and we are back to the case $q=1$. Suppose now $\dim_\mathbb C(E)=3$. Then $H_{1,q}|_E$ is of signature $(1,2)$. Let $E_{12}=span_\mathbb C(v_1,v_2)$. Then, the signature of $H_{1,q}|_E$ is again $(1,1)$. Hence, the orthogonal complement of $E_{12}$ in $E$, denoted by $N_{12}$, is complementary to $E_{12}$ and $\dim_\mathbb C(N_{12})=1$. Therefore, $E=E_{12}\oplus N_{12}$. Note that since $\|Av_j\|^2_{p,q}=\|v_j\|_{p,q}^2=0$ for $j=1,2$, by Lemma~\ref{semipositive}, we see that $N_{12}=v_1^\perp\cap v_2^\perp$ is invariant by $A$ and thus is a one dimensional eigenspace of $A$. Choose now an eigenvector $v_4\in N_{12}$. But on the other hand, since $H_{1,q}|_E$ is of signature $(1,2)$, we see that $H_{1,q}|_{N_{12}}$ is negative definite. This implies that $v_3$ and $v_4$ are not proportional. We now have four eigenvectors $v_1,v_2,v_3,v_4\in E$ which are pairwise non-proportional. But $\dim_\mathbb C(E)=3$, so this is possible only if  the restriction of $A$ on $E$ has at most two different eigenvalues. Suppose $v_j$ and $v_k$ are associated to the same eigenvalue $\lambda$ (which must be non-zero), then for $E_{jk}:=span_\mathbb C\{v_j,v_k\}$, we have that $H_{1,q}|_{E_{jk}}$ is of signature $(1,1)$ and hence we get an eigenvector $v\in E_{jk}$ which is also a positive vector and the desired result now follows.

For the case of $D_{p,1}$ with $p\geq 2$, we similarly let $G$ be a linear self map of $D_{p,1}$ with at least two fixed points in $\partial D_{p,1}$ and let $B$ be a matrix representation of $G$. By the hypotheses, we can find two null vectors $u_1,u_2\in\mathbb C^{p,1}$ such that they are not proportional and are eigenvectors of $B$ associated to some non-zero eigenvalues. Let $F_{12}:=span_\mathbb C(u_1,u_2)$. Again, since $\|Bu_j\|^2_{p,q}=\|u_j\|_{p,q}^2=0$ for $j=1,2$, we see from Lemma~\ref{semipositive} that $Q_{12}:=F_{12}^\perp=u_1^\perp\cap u_2^\perp$ is invariant by $B$. The signature of $H_{p,1}|_{F_{12}}$ is $(1,1)$ and thus $H_{p,1}|_{Q_{12}}$ is of signature $(p-1,0)$. But $\dim_\mathbb C(Q_{12})=p-1$ and hence $Q_{12}$ is a positive definite invariant subspace of $A$. It now follows that we can find an eigenvector of $A$ which is a positive vector. Therefore, we get a fixed point in $D_{p,1}$.

\end{proof}


We can now generalize our result to an arbitrary generalized ball.

\begin{thm}\label{bb generalization}
Let $F$ be a linear self map of $D_{p,q}$. Suppose there is no projective line in $\partial D_{p,q}$ on which every point is fixed by $F$.
\begin{enumerate}
\item If $p\leq q$ and $F$ has at least $2p+1$ fixed points on $\partial D_{p,q}$, then $F$ must have a fixed point in $D_{p,q}$. Furthermore, there is a projective line fixed everywhere by $F$.
\item  If $p>q$ and $F$ has at least $2q$ fixed points on $\partial D_{p,q}$, then $F$ must have a fixed point in $D_{p,q}$.
\end{enumerate}

\end{thm}

\noindent\textbf{Remark.} Note that every linear self map of the ball (or its exterior) satisfies the condition given in Theorem~\ref{bb generalization} 
since there is no projective line in the boundary.
However, in the case of $D_{p,q}$ with $p,q>1$, the condition is necessary. For instance, there is an automorphism
of $D_{2,2}$ which has infinitely many fixed points on the boundary but no fixed point in $D_{2,2}$ (Example~\ref{hs example}). 
Furthermore, the number $2p+1$ in Theorem~\ref{bb generalization} (1) is sharp since there is an automorphism having four fixed points on $\partial D_{2,2}$ and no fixed point in $D_{2,2}$ and there is no projective line in $\partial D_{2,2}$ on which every point is fixed (Example~\ref{sharp example 2}).

\begin{proof}[Proof of Theorem~\ref{bb generalization}]
Let $A$ be a matrix representation of $F$. 
As explained before Theorem~\ref{ball}, we can assume that $F$ is non-minimal. Thus, we can choose $A$ such that $A^H H_{p,q}A-H_{p,q}\geq 0$ by Theorem~\ref{expansion matrix}.

Let $s\geq 2$ be a positive integer and for every $k\in\{1,\ldots,s\}$, let $v_k\in\mathbb C^{p,q}$ be a null vector which is also an eigenvector of $A$ associated to a non-zero eigenvalue $\lambda_k$. Suppose that $\{v_1,\ldots,v_s\}$ are linearly independent. Thus, $\{[v_1],\ldots,[v_s]\}\subset\partial D_{p,q}$ are $s$ distinct boundary fixed points of $F$. For $i\neq j$, define $E_{ij}:=span_\mathbb C\{v_i, v_j\}\subset\mathbb C^{p,q}$.

 Assume that $\lambda_k=\lambda_\ell$ for some $k\neq\ell$. Then, as $v_k$, $v_\ell$ are two linearly independent null vectors, $H_{p,q}|_{E_{k\ell}}$ must be of signature $(1,1)$ or $E_{k\ell}$ is isotropic. In the first situation, we can find a positive vector in $E_{k\ell}$ while in the latter situation, the projectivization of $E_{k\ell}$ gives a projective line in $\partial D_{p,q}$. But $E_{k\ell}$ is an eigenspace of $A$ associated to $\lambda_k=\lambda_\ell$, so we either get a fixed point of $A$ in $D_{p,q}$ or we get a projective line in $\partial D_{p,q}$ on which every point is fixed by $A$.

So now we only need to consider the case where $\lambda_1,\ldots,\lambda_s$ are distinct.\,\,\,\,\,\,\,\,\,\,\,\,\,\,\,\,\,\,\,\,\,\,\,\,\,\,\,\,\,\,$(\star)$

We now divide the proof into two cases: 
\begin{enumerate}
\item Every $E_{ij}$ is isotropic; \label{i}
\item At least one $E_{ij}$ is not isotropic. \label{ii}
\end{enumerate}

In case \eqref{i}, $v_i$ is orthogonal (with respect to $H_{p,q}$) to $ v_j$ for every $i,j$ and it follows that $span_\mathbb C \{ v_1,\ldots, v_{s}\}$ is isotropic.
If $s\geq 2p+1$ or $s\geq 2q$, we get a contradiction since the maximal isotropic subspace in $\mathbb C^{p,q}$ is of dimension $\min(p,q)$.

In case \eqref{ii}, without loss of generality, we may assume that $E_{12}$ is non-isotropic. 
The restriction of $H_{p,q}$ on $E_{12}$ must be of signature $(1,1)$ (hence non-degenerate) as the null vectors $v_1$ and $v_2$ are linearly independent. 
 Now take any $v_k$, where $k\neq 1, 2$. Let $E_{12k}:=span_\mathbb C\{v_1, v_2, v_k\}$, which is a 3-dimensional invariant subspace of $A$. Let $N$ be the orthogonal complement of $E_{12}$ in $E_{12k}$. Then $N\cap E_{12}=\{0\}$ since $H_{p,q}|_{E_{12}}$ is non-degenerate. In particular, $\dim_\mathbb C(N)=1$. Moreover, as $\|Av_1\|_{p,q}^2=\|v_1\|_{p,q}^2=\|Av_2\|_{p,q}^2=\|v_2\|_{p,q}^2=0$, by Lemma \ref{semipositive}, $N=v_1^\perp\cap v_2^\perp$ is invariant under $A$. Hence, $N$ is a one-dimensional eigenspace of $A$. If $v_k\not\in N$, then we get four distinct one-dimensional eigenspaces in $E_{12k}$. Since $\dim_\mathbb C(E_{12k})=3$, at least three of these one-dimensional eigenspaces are associated to the same eigenvalue. 
It implies that at least two elements of $\{\lambda_1,\lambda_2,\lambda_k\}$ are equal, which contradicts ($\star$).

Finally, we just need to settle the situation where $v_k$ belongs to the orthogonal complement of $E_{12}$ for every $k\neq 1,2$.
Let $N_{12}$ be the orthogonal complement of $E_{12}$ in $\mathbb C^{p,q}$. Then, as before, since $H_{p,q}|_{E_{12}}$ is non-degenerate, we have $\mathbb C^{p,q}=E_{12}\oplus N_{12}$ and hence the restriction of $H_{p,q}$ on $N_{12}$ is of signature $(p-1,q-1)$. As argued previously, $N_{12}$ is an invariant subspace of $A$ by Lemma~\ref{semipositive}.
Furthermore, $[N_{12}]\cap D_{p,q}\cong D_{p-1,q-1}$, where $[N_{12}]$ denotes the projectivization of $N_{12}$ and for every $k\neq 1,2$, we have $[v_k]\in [N_{12}]\cap\partial D_{p,q}\cong\partial D_{p-1,q-1}$. That is, we have $s-2$ fixed points on $\partial D_{p-1,q-1}$ and the desired result now follows by induction together with Theorem \ref{ball}, which serves as the initial step for the induction. The proof is complete.
\end{proof}

\subsection{Real generalized balls}

Let $D_{p,q}^\mathbb R$ be the subspace of $D_{p,q}$
defined by 
$$
D_{p,q}^\mathbb R = \left\{ [x_1,\ldots, x_{p+q}]\in D_{p,q} : x_i/x_j\in \mathbb R \, \text{ for all } i,j \textrm{ whenever } x_j\neq 0\,\right\}.
$$

Thus, for a point in $D_{p,q}^\mathbb R$, we can choose homogeneous coordinates which are all real. We call $D_{p,q}^\mathbb R$ a \textit{real generalized ball}.
When $p=1$, Cowen-MacCluer \cite[Theorem 10]{CM} proved that any linear fractional map with real coefficients
maps $D_{1,q}^\mathbb R$ into itself if and only if it maps $D_{1,q}$ into itself. We now show that the same holds true for any $p\geq 1$.

\begin{thm}\label{real generalized ball}
Let $M \in M_{p+q}$ be a matrix with real entries.
Then $M$ defines a linear self map of $D_{p,q}$ if and only if it defines a linear self map of $D_{p,q}^\mathbb R$ in the same manner.
\end{thm}
\begin{proof}
Suppose that $M$ defines a linear self map of $D_{p,q}$.
Since the entries of $M$ are all real, $D_{p,q}^\mathbb R$ is mapped into $D_{p,q}^\mathbb R$.

Suppose that $M$ defines a linear self map of $D_{p,q}^\mathbb R$. For $[z]\in D_{p,q}$, write $z = x+iy$ with
$x=(x', x'')\in\mathbb R^{p+q}$, $y=(y', y'')\in\mathbb R^{p+q}$
for $x'$, $y'\in \mathbb R^p$, $x''$, $y''\in \mathbb R^q$.
Note that $|x'|^2 + |y'|^2 > |x''|^2 + |y''|^2$.
If $|x'|^2 >|x''|^2$ and $|y'|^2 >|y''|^2$, then $Mz = Mx+iMy$ belongs to $D_{p,q}$.
If otherwise, without loss of generality, we may assume that
$|x'|^2 \leq|x''|^2$ and $|y'|^2  > |y''|^2$.
Now, for any $\theta\in\mathbb R$, we have $[e^{i\theta}z]=[z]$. And since $e^{i\theta}z = x \cos\theta  - y \sin\theta + i(y\cos\theta + x\sin\theta) $ and the expression
\begin{equation}\label{real part}
|x'\cos\theta - y'\sin\theta|^2 - |x'' \cos\theta - y'' \sin\theta|^2
\end{equation}
equals $|x'|^2 - |x''|^2\leq 0$ when $\theta=0$ and
$|y'|^2 - |y''|^2> 0$ when $\theta=\pi/2$, by continuity there exists
$\theta$ such that equation \eqref{real part} becomes zero.
This implies that for $[z]\in D_{p,q}$, we can always choose homogeneous coordinates $z=x+iy$ such that $|x'|^2 = |x''|^2$ and $|y'|^2 > |y''|^2$. Here the last inequality is due to the fact that $e^{i\theta }z\in D_{p,q}$.
Hence we have $Mz\in D_{p,q}$ and in particular $M$ induces a linear self map of $D_{p,q}$.
\end{proof}


\section{Examples}\label{examples}

In this section, we collect a number of examples that the reader have been referred to from various places in the article.

\begin{example}\label{no extension}
Let $F$ be the rational map on $\mathbb P^3$ defined by $F[z_1,z_2,z_3,z_4]=[z_1+z_3, z_4, 0, 0]$. Then the indeterminacy of $F$ is the projective subspace spanned by $[1,0,-1,0]$ and $[0,0,0,1]$. As elements in $\mathbb C^{2,2}$, these two vectors spanned a negative semi-definite 2-dimensional subspace and hence we see that the set of indeterminacy of $F$ intersects $\partial D_{2,2}$ at $[1,0,-1,0]$ but lies outside $D_{2,2}$. Thus, $F$ is a holomorphic on $D_{2,2}$ but it cannot extend across the boundary point $[1,0,-1,0]$. It is a minimal linear self map of $D_{2,2}$ because the range of $F$ (as a linear map on $\mathbb C^{2,2}$) is positive definite and of dimension 2. 

\end{example}

\begin{example}\label{hs example}
Let $$A = \left(
           \begin{array}{cc}
             \sqrt{2}I_2& I_2 \\
             I_2 & \sqrt{2}I_2\\

           \end{array}
         \right)\in U(2,2).$$
$A$ induces an automorphism of $D_{2,2}$. The characteristic polynomial of $A$ is $(x-\sqrt{2}+1)^2(x-\sqrt{2}-1)^2$ and $A$ has two eigenvalues $\sqrt{2} \pm 1$. 
The eigenspace of the eigenvalue $\sqrt{2}-1$ is spanned by $v_1=(1,0,-1,0)^t$ and $v_2 = (0,1,0,-1)^t$ and that of $\sqrt{2}+1$ is spanned by $v_3=(1,0,1,0)^t$ and $v_4=(0,1,0,1)^t$. One also sees immediately that both eigenspaces are isotropic in $\mathbb C^{2,2}$ and thus their projectivizations lie inside $\partial D_{2,2}$. This implies that $A$ has infinitely many fixed points on the boundary but no fixed point in $D_{2,2}$. 
This example can be generalized to $D_{p,p}$ in a straightforward way.
\end{example}

\begin{example}\label{sharp example 3}
Let	$$A = \left(
           \begin{array}{cccc}
             1& 0 & 0& 0 \\
            0 & -1 & 0 & 0\\
             0 & 0 & i & 0\\
             0 &0&0& -i\\
           \end{array}
         \right)\in U(2,2).$$
Trivially $A$ has four different eigenvalues and two of them correspond to fixed points in $D_{2,2}$
and the other two of them corresponds to fixed points in $\mathbb P^3 \setminus \overline{D_{2,2}}$. There is no projective line on which every point is fixed.
\end{example}

\begin{example}\label{sharp example 2}
Let	$$A = \left(
           \begin{array}{cccc}
             1& 1 & 1& 0 \\
            1 & -1 & 0 & 1\\
             1 & 0 & 1 & 1\\
             0 &1&1& -1\\
           \end{array}
         \right)\in U(2,2).$$

The matrix $A$ has four different eigenvalues. The eigenvalues and the corresponding eigenvectors are the following:
\begin{equation}
\begin{array}{ll}
\lambda_1=1+\sqrt{2}, & \alpha_1 =(\frac{1}{4}+\frac{\sqrt{2}}{8},\frac{\sqrt{2}}{8},\frac{1}{4}+\frac{\sqrt{2}}{8},\frac{\sqrt{2}}{8})^t \\
\lambda_2=-1+\sqrt{2},& \alpha_2=(\frac{1}{4}+\frac{\sqrt{2}}{8},\frac{\sqrt{2}}{8},-\frac{1}{4}-\frac{\sqrt{2}}{8},-\frac{\sqrt{2}}{8})^t \\
\lambda_3=1-\sqrt{2},& \alpha_3=(\frac{1}{4}-\frac{\sqrt{2}}{8},-\frac{\sqrt{2}}{8}, \frac{1}{4}-\frac{\sqrt{2}}{8},-\frac{\sqrt{2}}{8})^t \\
\lambda_4=-1-\sqrt{2},& \alpha_4=(\frac{1}{4}-\frac{\sqrt{2}}{8},-\frac{\sqrt{2}}{8},-\frac{1}{4}+\frac{\sqrt{2}}{8},\frac{\sqrt{2}}{8})^t \\
\end{array}
\end{equation}
The automorphism of $D_{2,2}$ induced from $A$  has four fixed points on $\partial D_{2,2}$ but it has no fixed point in $D_{2,2}$. Moreover there is no projective line in $\partial D_{2,2}$
on which $A$ fixes every point. 
\end{example}


\begin{example} \label{nondiagonalizable}
For any non-zero real number $\alpha$ let $$A = \left(
           \begin{array}{cccc}
             1& \alpha & 0& \alpha \\
             -\alpha & 1 & \alpha & 0\\
             0 & \alpha & 1 & \alpha\\
             \alpha &0&-\alpha& 1\\
           \end{array}
         \right)\in U(2,2).$$
Then the characteristic polynomial of $A$ is $(1-x)^4$ and the minimal polynomial of $A$ is $(1-x)^2$. Thus, in any Jordan canonical form, the Jordan blocks of $A$ are at most of rank 2. By direct computation, one sees that $A$ only has two linearly independent eigenvectors, which are chosen to be $ (1,0,1,0)^t$ and $(0,1,0,-1)^t$, associated to the eigenvalue 1. In particular, there are two Jordan blocks of rank 2.

\end{example}


\end{document}